\documentclass[12pt]{amsart}
\usepackage{graphicx}
\usepackage{amssymb}
\usepackage{amsmath}
\usepackage{amscd}
\usepackage{upgreek}
\usepackage{mathrsfs}
\usepackage{stmaryrd}
\usepackage{longtable}
\usepackage{txfonts}
\usepackage[T1]{fontenc}
\usepackage{eucal}
\usepackage{titletoc}
\usepackage{wrapfig}
\usepackage{float}
\usepackage{xypic}
\usepackage{dsfont}
\usepackage{xcolor}
\usepackage{color}
\usepackage[colorlinks,linkcolor=blue,anchorcolor=blue,urlcolor=blue,
citecolor=blue]{hyperref}
\usepackage{hyperref}
\usepackage[hmarginratio=1:1,top=32mm,columnsep=20pt]{geometry}
	
\allowdisplaybreaks	
\usepackage{mathptmx}
\DeclareMathAlphabet{\mathcal}{OMS}{cmsy}{m}{n}
\DeclareSymbolFont{largesymbols}{OMX}{cmex}{m}{n}

\newtheorem{theorem}{Theorem}[section]
\newtheorem{proposition}[theorem]{Proposition}
\newtheorem{definition}[theorem]{Definition}
\newtheorem{lemma}[theorem]{Lemma}
\newtheorem{corollary}[theorem]{Corollary}
\newtheorem{thm}[theorem]{Theorem}
\newtheorem{remark}[theorem]{Remark}
\newtheorem{example}[theorem]{Example}
\newtheorem{prob}[theorem]{Problem}

\newcommand{\C}{\mathbb{C}}
\newcommand{\co}{\mathcal{O}}
\newcommand{\af}{\mathds{A}}
\newcommand{\N}{\mathbb{N}}

\newcommand{\F}{\mathbb{F}}

\newcommand{\LL}{\mathbb{L}}

\newcommand{\LA}{\Longleftarrow}
\newcommand{\RA}{\Longrightarrow}

\newcommand{\ra}{\longrightarrow}

\newcommand{\GL}{{\rm GL}}
\newcommand{\SL}{{\rm SL}}
\newcommand{\LM}{{\rm LM}}
\newcommand{\de}{{\rm def}}

\newcommand{\cha}{{\rm char}}
\newcommand{\rank}{{\rm rank}}

\newcommand{\sing}{{\rm Sing}}
\newcommand{\dep}{{\rm depth}}
\newcommand{\var}{{\rm Var}}

\newcommand{\st}{{\rm st}}

\newcommand{\age}{{\rm age}}

\newcommand{\conj}{{\rm Conj}}

\newcommand{\hbo}{\hfill\Box}

\begin{document}
\setlength{\oddsidemargin}{0cm}
\setlength{\evensidemargin}{0cm}

%
\title{\scshape Modular quotient varieties and singularities by the cyclic group of order $2p$}

\author{\scshape Yin Chen}
\address{School of Mathematics and Statistics, Northeast Normal University, Changchun 130024, P.R. China}
\email{ychen@nenu.edu.cn}

\author{\scshape Rong Du$^{\dag}$}
\address{School of Mathematical Sciences\\
Shanghai Key Laboratory of PMMP\\
East China Normal University\\
Rm. 312, Math. Bldg, No. 500, Dongchuan Road\\
Shanghai, 200241, P. R. China}
\email{rdu@math.ecnu.edu.cn}

\author{\scshape Yun Gao}
\address{Department of Mathematics, Shanghai Jiao Tong University,
Shanghai 200240, P. R. of China}
\email{gaoyunmath@sjtu.edu.cn}

\date{\today}
\def\shorttitle{Modular quotient varieties and singularities}

\thanks{$^{\dag}$ Corresponding author.}

%
\begin{abstract}
We  classify all $n$-dimensional reduced Cohen-Macaulay modular quotient variety $\af^n_{
\F}/C_{2p}$ and study their singularities, where $p$ is a prime number and $C_{2p}$ denotes the cyclic group of order $2p$. In
particular, we present an example that demonstrates that the problem proposed by Yasuda \cite[Problem 6.6]{Yas2015}
has a negative answer if the condition that ``$G$ is a small subgroup'' was dropped. 
\end{abstract}


\maketitle
\baselineskip=15pt


\dottedcontents{section}[1.16cm]{}{1.8em}{5pt}
\dottedcontents{subsection}[2.00cm]{}{2.7em}{5pt}
\dottedcontents{subsubsection}[2.86cm]{}{3.4em}{5pt}



\section{\scshape Introduction}

Let $\F$ be an algebraically closed field of any characteristic  and
$\af_{\F}^{n}$ be the $n$-dimensional affine space over $\F$.
Suppose a finite group $G$ acts on $\af_{\F}^{n}$, linearly but not necessarily faithfully.
Then the coordinate ring of the quotient variety $\af_{\F}^{n}/G$ is isomorphic to the invariant ring $\co(\af_{\F}^{n})^{G}$ which is the main object of study in Algebraic Invariant Theory. The groundbreaking theorem which says that $\co(\af_{\F}^{n})^{G}$ is always a finitely generated $\F$-algebra, was proved by Emmy Noether in 1926. Since then, the study on relations  between 
the algebraic structure of $\co(\af_{\F}^{n})^{G}$ and the geometry of $G$ on $\af_{\F}^{n}$ has been occupied the central position in the invariant theory of finite groups; see Campbell-Wehlau \cite{CW2011}, Derksen-Kemper \cite{DK2015} and Neusel-Smith \cite{NS2002} for general references.
We say that the quotient variety $\af_{\F}^{n}/G$ is \textit{modular} if the characteristic of $\F$ divides the order of $G$; otherwise \textit{nonmodular}. The affine space $\af_{\F}^{n}$ could be regarded as a linear representation of $G$, and  we call $\af_{\F}^{n}/G$ \textit{reduced} if any direct summand of $\af_{\F}^{n}$ is not isomorphic to the one-dimensional trivial representation of $G$.
We also say that $\af_{\F}^{n}/G$ is a \textit{Cohen-Macaulay} (\textit{Gorenstein, complete intersection, hypersurface}) quotient if $\co(\af_{\F}^{n})^{G}$ is \textit{Cohen-Macaulay} (\textit{Gorenstein, complete intersection, hypersurface} respectively). We use $\sing_{G}(\af_{\F}^{n})$ to denote the set of all singularities in $\af_{\F}^{n}/G$.

Suppose $\af_{\F}^{n}/G$ is nonmodular, i.e.,  the characteristic of $\F$ is zero or a prime $p$ such that $p$ does not divide the order of $G$. The famous Shephard-Todd-Chevalley theorem says that $\sing_{G}(\af_{\F}^{n})$ is empty (i.e., $\af_{\F}^{n}/G$ is smooth) if and only if the action of $G$ on $\af_{\F}^{n}$ is a group generated by reflections; see Shephard-Todd \cite{ST1954} and Chevalley \cite{Che1955}.  Noether's bound theorem demonstrates that any invariant ring $\co(\af_{\F}^{n})^{G}$ can be generated by polynomial invariants of degree less than or equal to the order of $G$; see for example Derksen-Kemper \cite[Corollary 3.2.4]{DK2015}.
Another important result due to Hochster and Eagon asserts that
$\af_{\F}^{n}/G$ is always Cohen-Macaulay; see Hochster-Eagon \cite{HE1971}.

For  modular quotient varieties  $\af_{\F}^{n}/G$, the situation becomes to be worse and more complicated.
A theorem of Serre shows that if $\af_{\F}^{n}/G$ is smooth, then the action of $G$ on $\af_{\F}^{n}$ is a group generated by reflections; see  \cite{Ser1968}. However, there exists a counterexample that demonstrates the converse statement of Serre's theorem is not feasible in general; see Campbell-Wehlau \cite[Section 8.2]{CW2011}. Noether's bound theorem  and Hochster-Eagon's theorem also failed to hold in the modular cases; see Richman \cite{Ric1990} and Campbell-Wehlau \cite[Example 8.0.9]{CW2011}.
For the low-dimensional case, however,  if $n\leqslant3$ then $\af_{\F}^{n}/G$ is always Cohen-Macaulay; see Smith \cite{Smi1996}.
Further, considering $G=C_{4}$  acting on $\af_{\F}^{4}$ regularly and $\cha(\F)=2$, Bertin proved that
$\af_{\F}^{4}/C_{4}$ is factorial but not Cohen-Macaulay, answering a question of Pierre Samuel; see  \cite{Ber1967}.
For recent explicit computations of modular Cohen-Macaulay quotient varieties, see Chen \cite{Che2014}, Chen-Wehlau \cite{CW2017} and  references therein. 

Motivated by connecting representations theory with crepant resolutions of quotient varieties, the  McKay correspondence (which, roughly speaking, reveals an equality between an invariant of representations of a finite subgroup $G\subset \SL_{n}(\C)$ and  an invariant of a crepant resolution of the quotient variety $\C^n/G$) in characteristic 0 has been studied extensively; for example, see 
Batyrev \cite{Bat1999} and Batyrev-Dais \cite{BD1996} in terms of stringy invariants, Denef-Loeser \cite{DL2002} from the theory of motivic integration, and Reid \cite{Rei2002}.  In 2014, Yasuda \cite{Yas2014} studied the McKay correspondence for modular representations of the cyclic group $C_{p}$ of order $p$ and calculated the stringy invariant of the quotient variety via the motivic integration generalized to quotient stacks associated to representations, leading to  a conjecture of Yasuda \cite[Conjecture 1.1]{Yas2015} with a positive answer for the modular cases of $C_{p}$. In particular, it was proved that the Euler characteristic of the crepant resolution variety is equal to the number of the conjugacy class of $C_{p}$. In his notes, Yasuda \cite[Problem 6.6]{Yas2015} asks whether this statement remains valid for 
any  small subgroup $G\subset \GL_{n}(\C)$ in the modular cases. Here a ``small'' subgroup $G$ means $G$ contains no any non-identity reflections. 

\begin{prob}\label{prob1.1}
{\rm
Let $G\subset\GL_n({\F})$ be a small subgroup and $Y\ra X=\af_{\F}^{n}/G$ be a crepant resolution. Suppose $\cha(\F)=p>0$ divides the order of $G$.  Is the Euler characteristic of $Y$ equal to the number of the conjugacy classes of $G$?
}\end{prob}

There are some works done on the McKay correspondence in the modular cases and for different purposes; see Gonzalez-Verdier \cite{GV1985}, Schr\"oer \cite{Sch2009}, and Yasuda  \cite{Yas2017}. However, to the best of our knowledge, no more results (even for finite cyclic groups) have appeared to answer Problem \ref{prob1.1} except  Yasuda's work \cite{Yas2014} on the cyclic group $C_{p}$.

The purpose of this paper is to study singularities in $n$-dimensional  Cohen-Macaulay modular quotient variety $\af_{\F}^{n}/C_{2p}$, where $p$ is a prime number and $C_{m}$ denotes the cyclic group of order $m\in\N^{+}$. In particular, we will provide an example that demonstrates that Problem  \ref{prob1.1} has a negative answer  in the modular cases of $C_{2p}$
if the condition that ``$G$ is a small subgroup'' was dropped.

To make $\af_{\F}^{n}/C_{2p}$ to be a modular quotient variety, we have the following three situations:
(1) char$(\F)=p=2$; (2) char$(\F)=p>2$; and (3) char$(\F)=2$ and $p>2$.

For the first case we prove the following result.

\begin{thm}\label{thm1.2}
Let $\af_{\F}^{n}/C_{4}$ be an $n$-dimensional, reduced  Cohen-Macaulay quotient variety  in characteristic 2. Then
\begin{enumerate}
  \item $n\leqslant 4$ and $\af_{\F}^{n}/C_{4}$ is always a hypersurface;
  \item $\af_{\F}^{n}/C_{4}$ is smooth if and only if $n=2$; in this case, $\af_{\F}^{2}/C_{4}$ is isomorphic to $\af_{\F}^{2}/C_{2}$;
  \item $\sing_{C_{4}}(\af_{\F}^{3})\cong \af_{\F}^{1}$ and $\sing_{C_{4}}(\af_{\F}^{4})\cong \af_{\F}^{2}$.
\end{enumerate}
\end{thm}

For the second case, we prove that there always exists a  reduced modular  Cohen-Macaulay quotient variety in each dimension $n$; and all  reduced modular Cohen-Macaulay quotient varieties can be realized via one of seven families of modular representations
of $C_{2p}$; see Proposition \ref{prop3.1}.
This is the ingredient of classifying all two and three-dimensional reduced modular Cohen-Macaulay quotient varieties; see Corollary \ref{coro3.2}. Furthermore, we describe all quotient singularities and varieties in dimensions two and three as follows.

\begin{thm}\label{thm1.3}
Let $\af_{\F}^{n}/C_{2p}$ be an $n$-dimensional, reduced  Cohen-Macaulay quotient variety  in characteristic $p>2$.
\begin{enumerate}
  \item $\af_{\F}^{2}/C_{2p}$ is always a hypersurface and $\sing_{C_{2p}}(\af_{\F}^{2})$ either is empty or consists of the zero point;
  \item $\af_{\F}^{3}/C_{2p}$ is either a hypersurface or not Gorenstein.
  \end{enumerate}
\end{thm}

\noindent We also prove that if $n\geqslant 4$ and $\af_{\F}^{n}$ is isomorphic to the $n$-dimensional indecomposable
modular representation of $C_{2p}$, then $\af_{\F}^{n}/C_{2p}$ is never Cohen-Macaulay; see Proposition \ref{prop3.5}.

For the third case, we prove that all reduced modular Cohen-Macaulay quotient varieties can be realized through
direct sums of finitely many 1-dimensional and at most two 2-dimensional indecomposable representations; see Proposition \ref{prop4.1}. We calculate the invariant ring $\co(\af_{\F}^{2})^{C_{2p}}$ for any reduced quotient variety $\af_{\F}^{2}/C_{2p}$
by giving explicit set of generators and relations between these generators; see Propositions \ref{HWp} and \ref{prop4.6}.
As a consequence, the following result

\begin{thm} \label{thm1.4}
For any $1\leqslant i,j\leqslant p-1$, $\F[V_{i}\oplus V_{j}]^{C_{2p}}$ is Gorenstein if and only if $i+j=p$.
\end{thm}
 
 \noindent together with Remark \ref{rem4.8} gives rise to a classification of all reduced, two-dimensional modular Gorenstein quotient varieties.

In Section 2, we consider the modular quotient varieties of $C_{2p}$ in characteristic $p$ and give the proofs of Theorem \ref{thm1.2} and Theorem \ref{thm1.3}. In Section 3, we study the reduced modular quotient varieties of $C_{2p}$ ($p>2$) in characteristic $2$, emphasizing a characterization  of $2$-dimensional  modular Gorenstein quotient varieties
$\af^2_{\F}/C_{2p}$ in Theorem \ref{thm1.4}. Section 4 contains a discussion on the McKay correspondence in the modular cases (or wild McKay correspondence) and we present an example to show the necessary of the assumption that $G$ is a small subgroup  in Problem \ref{prob1.1}.

\section{\scshape Modular Quotient Varieties of $C_{2p}$ in Characteristic $p$} \label{section2}

Let $W$ be a finite-dimensional representation of a group $G$ over $\F$. Recall that a non-identity  invertible linear map $\sigma: W\ra W$ is called a \textbf{bi-reflection} if the dimension of the image of $\sigma-1$
is less than or equal to 2. We say that $G$ is a \textbf{bi-reflection group} on its representation $W$ if the image of $G$ under the group homomorphism $G\ra \GL(W)$ is a group generated by bi-reflections.
Let $\F[W]$ denote the symmetric algebra on the dual space $W^{*}$ of $W$. Choose a basis $e_{1},\dots,e_{n}$ for $W$ and a dual basis $x_{1},\dots,x_{n}$ in $W^{*}$. Then $\F[W]$ can be identified with the polynomial algebra $\F[x_{1},\dots,x_{n}]\cong 
\co(\af_{\F}^{n})$.
For a finitely generated graded subalgebra $A\subseteq \F[W]$, let $A_{+}$ denote its maximal homogeneous ideal.
The depth of a homogeneous ideal $I\subseteq A_{+}$, written $\dep_{A}(I)$, is the maximal length of a regular sequence in $I$.
In particular, we write $\dep(A)=\dep_{A}(A_{+})$.
We define $\de(A):=\dim(A)-\dep(A)$ to be the \textbf{Cohen-Macaulay defect} of $A$. Clearly, $A$ is Cohen-Macaulay if and only if 
$\de(A)=0$. In fact, the Cohen-Macaulay defect of $A$ measures how far $A$ is from being Cohen-Macaulay. 

The following two results are extremely useful to decide whether a modular invariant ring is a Cohen-Macaulay algebra.

\begin{thm}[Kemper \cite{Kem1999}] \label{kemper99}
Let $W$ be a finite-dimensional modular representation of a finite $p$-group $P$ over $\F$ of characteristic $p>0$. If the invariant ring $\F[W]^{P}$ is a Cohen-Macaulay algebra, then $P$ is a bi-reflection group on $W$.
\end{thm}

\begin{thm}[Ellingsrud-Skjelbred \cite{ES1980} or Kemper \cite{Kem2012}] \label{ES}
Let $G$ be a cyclic group  with Sylow $p$-subgroup $P$ and $W$ be any finite-dimensional modular representation of $G$ over a field $\F$ of characteristic $p>0$. Then $\de(\F[W]^G)=\max\{\textrm{codim}(W^P)-2,0\}$.
\end{thm}

Now we give a proof of Theorem \ref{thm1.2}.

\begin{proof}[Proof of Theorem \ref{thm1.2}] Let $C_{4}$ be generated by $\sigma$ and $V$ be any finite-dimensional indecomposable representation of $C_{4}$ over $\F$. It is well-known that  $V$ must be isomorphic to one of $\{V_{k}\mid 1\leqslant k\leqslant 4\}$, where $\dim(V_{k})=k$ and $V_{k}$ corresponds to the group homomorphism $C_{4}$ to $\GL(V_{k})$ defined by carrying $\sigma$ to the $k\times k$-Jordan block with the diagonals 1; see for example Campbell-Wehlau \cite[Lemma 7.1.3]{CW2011}. More precisely, $V_{1}$ is the trivial representation and $V_{4}$ is the regular one; $V_{1}$ and $V_{2}$ are not faithful meanwhile $V_{3}$ and $V_{4}$ are faithful.
As  $\af_{\F}^{n}/C_{4}$ is reduced,  $\af_{\F}^{n}$ contains no the trivial representation $V_{1}$ as direct summand. Thus there exist $a,b,c\in\N$ such that
$\af_{\F}^{n}\cong aV_{2}\oplus b V_{3}\oplus c V_{4}.$
This implies that the dimension of the image of $\sigma-1$ is $a+2b+3c$.
The fact that  $\af_{\F}^{n}/C_{4}$  is Cohen-Macaulay, together with Theorem \ref{kemper99}, forces $a+2b+3c$ to be less than or equal to 2.
Hence, $c$  must be zero and clearly, there are only three possibilities: (i) $b=1$ and $a=0$; (ii) $b=0$ and $a=1$; and (iii)
$b=0$ and $a=2$. This proves $n\leqslant 4$ and $\af_{\F}^{n}/C_{4}$ must be isomorphic to one of $\{V_{3}/C_{4},V_{2}/C_{4},
2V_{2}/C_{4}\}$.

Magma calculations show that the invariant ring $\co(V_{3})^{C_{4}}=\F[x,y,z]^{C_{4}}$ is a hypersurface, minimally generated by
$f_{1}  :=  x, f_{2}  :=   xy+y^{2}, f_{3} := x^2 yz + x^2 z^2 + xy^2 z + xyz^2 + y^2 z^2 + z^4$, and $h := x^2 z + xy^2 + xz^2 + y^3
$ subject to the unique relation: $f_{1}^2 f_{3} + f_{1}f_{2}h + f_{2}^3 + h^2=0.$ Hence
$\co(V_{3}/C_{4})\cong \F[x_{1},\dots,x_{4}]/(f)$ where $f:=x_{1}^2 x_{3} + x_{1}x_{2}x_{4} + x_{2}^3 + x_{4}^2.$
Note that characteristic of $\F$ is 2 and $\sing_{C_{4}}(V_{3})=\{e\in \af_{\F}^{4}\mid f(e)=0\textrm{ and } \rank(\frac{\partial f}{\partial x_{1}}(e),\dots,
\frac{\partial f}{\partial x_{4}}(e))=0\}$. A direct calculation shows that
$\sing_{C_{4}}(V_{3})=\{(0,0,a,0)\in \af_{\F}^{4}\mid a\in\F\} \cong \af_{\F}^{1}$.
Similarly, as $V_{2}$ is not a faithful representation of $C_{4}$ and the image of $C_{4}$ on $V_{2}$ is actually isomorphic to the image of $C_{2}$ acting faithfully on $V_{2}$. Thus  $\co(V_{2}/C_{4})=\co(V_{2}/C_{2})$ and
$\co(2V_{2}/C_{4})=\co(2V_{2}/C_{2})$. These two invariant rings have already been calculated  in Campbell-Wehlau \cite[Theorem 1.11.2]{CW2011}
and \cite[Theorem 1.12.1]{CW2011} respectively. Indeed, $\co(V_{2}/C_{4})=\F[x,y]^{C_{2}}=\F[x,xy+y^{2}]$ is a polynomial algebra
and $\co(2V_{2}/C_{4})=\F[x_{1},y_{1},x_{2},y_{2}]^{C_{2}}$ is a hypersurface, minimally generated by
$\{x_{i},N_{i}:=x_{i}y_{i}+y_{i}^{2},u:=x_{1}y_{2}+x_{2}y_{1}\mid i=1,2\},$
subject to the unique relation: $u^{2}=x_{1}^{2}N_{2}+x_{2}^{2}N_{1}+x_{1}x_{2}u.$
This yields that $\af_{\F}^{n}/C_{4}$ is always a hypersurface. Clearly, $V_{2}/C_{4}$ is smooth and
a direct calculation similar with the case $V_{3}$ applies for the case $2V_{2}$. Eventually we derive
$\sing_{C_{4}}(2V_{2})\cong \af_{\F}^{2}$. This completes the proof.
\end{proof}

To accomplish the proof of Theorem \ref{thm1.3},  we suppose char$(\F)=p>2$ throughout the rest of this section. Then $C_{2p}$ is isomorphic to the direct product of $C_{2}$ and
$C_{p}$.
We have seen in Alperin  \cite[pages 24--25]{Alp1993} that up to equivalence, there are only $2p$ indecomposable modular representations $\{V_{k}^{\pm} \mid 1\leqslant k\leqslant p\}$ of $C_{2p}$, where $\dim(V_{k}^{\pm})=k$ for all $k$. More precisely, for $1\leqslant k\leqslant p$, $V_{k}^+$ corresponds to the group homomorphism carrying $\sigma$ to $T^+_k$, the $k\times k$ Jordan block with diagonals 1; and $V_{k}^-$ corresponds to the group homomorphism carrying $\sigma$ to $T^-_k$, the $k\times k$ Jordan block  with diagonals $-1$.
Note that the orders of $T^+_k$ and $T^-_k$
are $p$ and $2p$ respectively for $k\geq 2$. Thus all $V_k^+$ are not faithful; and all $V_k^-$ are faithful except for $k=1$.
Further, $(T^-_k)^2$ and $T^+_k$ generate the same $p$-sylow subgroup $P$ of $C_{2p}$ where $|P|=p$.
Hence, the invariant ring $\F[V_k^-]^{C_{2p}}$ can be viewed to be a subring of $\F[V_k^+]^{C_{2p}}$ via $\F[V_k^-]^{C_{2p}}\subseteq \F[V_k^-]^{P}\cong\F[V_k^+]^{C_{2p}}$.

\begin{proposition}\label{prop3.1}
Suppose char$(\F)=p>2$ and $W$ is a finite-dimensional reduced modular representation of $C_{2p}$ over $\F$ such that
$\F[W]^{C_{2p}}$ is Cohen-Macaulay. If $W$ is faithful, then it must be isomorphic to one of $\{V_2^+\oplus b V_1^-, 2V_2^+\oplus b V_1^-, V_3^+\oplus b V_1^-, b_1 V_1^-\oplus V_2^-, b_1 V_1^-\oplus 2V_2^-, b_1 V_1^-\oplus V_3^-, V_2^+\oplus V_2^-\oplus b_{1}V_1^-\}$ where $b\in\N^+$ and $b_1\in \N$; if $W$ is not faithful,  then it must be isomorphic to one of $\{b V_1^-, V_2^+, 2V_2^+, V_3^+\}$
where $b\in\N^+$.
\end{proposition}

\begin{proof}
As $W$ is reduced, we may suppose $a_2,\dots,a_p,b_1,\dots,b_p\in\N$ such that
$W\cong a_2V_2^+\oplus \cdots \oplus a_pV_p^+\oplus b_1V_1^-\oplus b_2V_2^-\oplus\cdots \oplus b_pV_p^-.$
Note that $C_{2p}$ is generated by $\sigma$ and $P$ is generated by $\sigma^2$, thus
$\textrm{codim}(W^P)=\dim(W)-\dim(W^P)$, where
$W^P=\{w\in W\mid \delta(w)=w,\textrm{ for all }\delta\in P\}=\{w\in W\mid \sigma^2(w)=w\}$.
Hence, $\textrm{codim}(W^P)=\rank(\sigma^2-1)$.
Since the rank of  $\sigma^2-1$ on $V_k^{\pm}$ is $k-1$, it follows that the rank of  $\sigma^2-1$ on $W$
is equal to $a_2+a_3 \cdot 2+\cdots+a_p\cdot (p-1)+b_1\cdot 0+b_2+b_3\cdot 2+\cdots+b_p\cdot (p-1)$.
Since $\F[W]^{C_{2p}}$ is Cohen-Macaulay, $\de(\F[W]^{C_{2p}})$ must be zero.  It follows Theorem \ref{ES} that $\textrm{codim}(W^P)=\rank(\sigma^2-1)\leq 2$, i.e., $$\sum_{k=2}^{p} (a_{k}+b_{k})(k-1)\leqslant 2.$$
This implies that $a_{k}=b_{k}=0$ for $k=4,5,\dots,p$; and so $a_{2}+b_{2}+2a_{3}+2b_{3}\leqslant 2.$
Clearly, $(a_{2}, b_{2}, a_{3}, b_{3})$ must be one of $$\{(1,0,0,0),(0,1,0,0), (1,1,0,0),(2,0,0,0), (0,2,0,0),(0,0,1,0),(0,0,0,1)\}.$$
Note that $b_{1}$ is arbitrary; so if $W$ is faithful, then it must be isomorphic to one of $\{V_2^+\oplus b V_1^-, 2V_2^+\oplus b V_1^-, V_3^+\oplus b V_1^-, b_1 V_1^-\oplus V_2^-, b_1 V_1^-\oplus 2V_2^-, b_1 V_1^-\oplus V_3^-, V_2^+\oplus V_2^-\oplus b_{1}V_1^-\}$ for some $b\in\N^+$ and $b_1\in \N$; if $W$ is not faithful,  then it must be isomorphic to one of $\{b V_1^-, V_2^+, 2V_2^+, V_3^+\}$ for some $b\in\N^+$.
This completes the proof.
\end{proof}

\begin{corollary}\label{coro3.2}
Let $\af_{\F}^{n}/C_{2p}$ be an $n$-dimensional, reduced  Cohen-Macaulay quotient variety  in characteristic $p>2$.
Then $\af_{\F}^{2}/C_{2p}$ is isomorphic to the quotient obtained from one of $\{V_2^+,2 V_1^-, V_2^-\}$; and $\af_{\F}^{3}/C_{2p}$ is isomorphic to the quotient obtained from one of $$\{V_2^+\oplus V_1^-,V_{3}^{+},3V_1^-,V_3^-,V_2^-\oplus V_1^-\}.$$
\end{corollary}

\begin{lemma}\label{lem3.3}
Let $V_{1}$ be the one-dimensional irreducible nontrivial representation of $C_{2}$ over $\F$ of characteristic $p>2$.
Then $\F[3V_{1}]^{C_{2}}=\F[x,y,z]^{C_{2}}$ is Cohen-Macaulay but not Gorenstein, minimally generated by $\{x^{2},y^{2},z^{2},xy,xz,yz\}$. Moreover,   $\F[3V_{1}]^{C_{2}}$ is isomorphic to the quotient algebra $\F[x_{1},x_{2},\dots,x_{6}]/(r_{1},r_{2},\dots,r_{6})$ where
  $r_{1}=x_{4}^{2}-x_{1}x_{2}, r_{2}=x_{5}^{2}-x_{1}x_{3}, r_{3}=x_{6}^{2}-x_{2}x_{3}, r_{4}=x_{4}x_{5}-x_{1}x_{6},
  r_{5}=x_{4}x_{6}-x_{2}x_{5}$ and $r_{6}=x_{5}x_{6}-x_{3}x_{4}$.
\end{lemma}

\begin{proof}
Note that the generator of $C_{2}$ sends $x,y$ and $z$ to $-x,-y$ and $-z$ respectively. Thus there are no invariants of degree 1.
By Noether's bound theorem, $\F[3V_{1}]^{C_{2}}$ can be generated by invariants of degree 2.
Since $x^{2},y^{2},z^{2},xy,xz,yz$ forms a basis of $\F[3V_{1}]_{d}$ and all are invariants, they minimally generate
 $\F[3V_{1}]^{C_{2}}$. By Hochster-Eagon \cite[Proposition 13]{HE1971} and Watanabe  \cite[Theorem 1]{Wat1974b} we see that
 $\F[3V_{1}]^{C_{2}}$ is Cohen-Macaulay but not Gorenstein. Magma's calculation proves the second statement.
\end{proof}

\begin{proof}[Proof of Theorem \ref{thm1.3}]
We first consider the two-dimensional case, assuming $\F[V_2^+]=\F[2 V_1^-]=\F[V_2^-]=\F[x,y]$. Since the image of $C_{2p}$ on $V_{2}^{+}$ is $P$, it follows that
$\F[V_{2}^{+}]^{C_{2p}}=\F[x,y]^{P}=\F[x,N(y)]$ is a polynomial algebra, where $N(y):=y^{p}-yx^{p-1}$.
Hence, $V_{2}^{+}/C_{2p}$ is smooth. We observe that $\F[2V_{1}^{-}]^{C_{2p}}=\F[x^{2},xy,y^{2}]\cong \F[z_{1},z_{2},z_{3}]/(z_{2}^{2}-z_{1}z_{3})$ and thus $\sing_{C_{2p}}(2V_{1}^{-})$ consists of the zero point in $\af_{\F}^{3}$.
Moreover, $\F[V_2^-]^{C_{2p}}=(\F[V_2^-]^{P})^{C_{2p}/P}\cong\F[x,N(y)]^{C_{2}}$ where the generator of $C_{2}$ sends
$x$ to $-x$ and $N(y)$ to $-N(y)$ respectively. Thus $\F[V_2^-]^{C_{2p}}=\F[x^{2},x\cdot N(y),N(y)^{2}]\cong \F[2V_{1}^{-}]^{C_{2p}}$. Hence, $\sing_{C_{2p}}(V_{2}^{-})$ also consists of the zero point in $\af_{\F}^{3}$. This proves the first statement.

For the three-dimensional case we suppose $\{x,y,z\}$ is the dual basis for the representation.
By Kemper \cite[Proposition 16]{Kem1996} we see that $\F[V_2^+\oplus V_1^-]^{C_{2p}}=\F[x,N(y),z^{2}]$.
Note that $\F[V_{3}^{+}]^{C_{2p}}=\F[V_{3}^{+}]^{C_{p}}$ and the latter has been calculated by Campbell-Wehlau
\cite[Section 4.10]{CW2011}. Thus $\F[V_{3}^{+}]^{C_{2p}}$ is a hypersurface, generated by $\{x,d,N^{P}(y),N^{P}(z)\}$
where $d=y^{2}-2xz-xy$ and $N^{P}(*)$ denotes the norm of $*$. As $\F[3V_{1}^{-}]^{C_{2p}}\cong \F[3V_{1}]^{C_{2}}$, it follows from Lemma \ref{lem3.3} that $\F[3V_{1}^{-}]^{C_{2p}}$ is Cohen-Macaulay but not Gorenstein, minimally generated by six invariants subject to six relations among these generators.
To calculate $\F[V_{3}^{-}]^{C_{2p}}$, we consider the subgroup $P$ of $C_{2p}$ and the quotient group $C_{2p}/P\cong C_{2}$. We have seen that $\F[V_{3}^{-}]^{P}=\F[V_{3}^{+}]^{C_{2p}}=\F[x,d,N^{P}(y),N^{P}(z)]$. Since
the generator of $C_{2}$ fixes $d$, and sends $x, N^{P}(y),N^{P}(z)$ to $-x, -N^{P}(y),-N^{P}(z)$ respectively, there exists a natural $C_{2}$-equivalent algebra surjection:
$$\rho:\F[x_{1},x_{2},x_{3},x_{4}]\ra \F[x,N^{P}(y),N^{P}(z),d]=\F[V_{3}^{-}]^{P}$$
carrying $x_{1}$ to $x$, $x_{2}$ to $N^{P}(y)$, $x_{3}$ to $N^{P}(z)$, and $x_{4}$ to $d$. Here
$\F[x_{1},x_{2},x_{3},x_{4}]^{C_{2}}\cong \F[3V_{1}]^{C_{2}}\otimes_{\F}\F[x_{4}]$ and we take $x_{1},x_{2},x_{3}$
as a basis of $(3V_{1})^{*}$ dual to the standard basis of $3V_{1}$. Since $C_{2}$ is a linearly reductive,
$\rho$ restricts to the following $\F$-algebra surjection:
$$\rho^{C_{2}}:\F[x_{1},x_{2},x_{3},x_{4}]^{C_{2}}\ra \F[x,N^{P}(y),N^{P}(z),d]^{C_{2}}=(\F[V_{3}^{-}]^{P})^{C_{2}}\cong
\F[V_{3}^{-}]^{C_{2p}}.$$
By Lemma \ref{lem3.3}, we see that $\F[V_{3}^{-}]^{C_{2p}}$ is generated by
$$\{d,x^{2},N^{P}(y)^{2},N^{P}(z)^{2}, xN^{P}(y),xN^{P}(z), N^{P}(y)\cdot N^{P}(z)\}.$$
Recall that $N^{P}(y)^{2}\in \F[x,d,N^{P}(z)]\oplus  N^{P}(y)\cdot \F[x,d,N^{P}(z)]$; see Campbell-Wehlau \cite[page 77]{CW2011}.
Hence, $\F[V_{3}^{-}]^{C_{2p}}$ is minimally generated by
$\{d,x^{2},N^{P}(z)^{2}, xN^{P}(y),xN^{P}(z), N^{P}(y)\cdot N^{P}(z)\}.$
Similar argument applies to the last subcase. Note that $\F[V_2^-\oplus V_1^-]^{P}=\F[x,N(y),z]$
and the generator of $C_{2p}/P$ carries $x,N(y),z$ to $-x,-N(y),-z$ respectively. Hence,
$\F[V_2^-\oplus V_1^-]^{C_{2p}}$ is minimally generated by $\{x^{2},N(y)^{2},z^{2},x\cdot N(y),xz,z\cdot N(y)\}$. Therefore, Smith
\cite[Theorem 1.2]{Smi1996}, together with Watanabe \cite[Theorem 1]{Wat1974b}, implies that $\F[V_{3}^{-}]^{C_{2p}}$ and $\F[V_2^-\oplus V_1^-]^{C_{2p}} (\cong \F[3V_{1}]^{C_{2}})$ are Cohen-Macaulay but not Gorenstein.
\end{proof}

We currently are unable  to characterize all singularities $\sing_{C_{2p}}(V_{3}^{-})$ even for the case $p=3$ but we can find out some partial subset of singularities  as the following example shows. 

\begin{example}{\rm Let $p=3$ and $v\in \sing_{C_{6}}(V_{3}^{-})$. Magma's calculation shows that the invariant ring $\F[V_{3}^{-}]^{C_{6}}=\F[x,y,z]^{C_{6}}$ is minimally generated by the primary invariants
$f_{1}:=x^2, f_{2}:=xy + xz + y^2, f_{3}:= x^2 y^2 z^2 + x^2 y z^3 + x^2 z^4 + 2x y^3 z^2 + xy^2 z^3 + xyz^4 +
        2xz^5 + y^4z^2 + y^2z^4 + z^6$ and the secondary invariants
   $ h_{1}:=x^3 z + x^2 y^2 + xy^3, h_{2}:=x^2 yz + 2x^2 z^2 + xy^2 z + 2xz^3,
   h_{3}:= x^4 yz + 2x^4 z^2 + x^3 y^2 z + x^3 y z^2 + x^3 z^3 + x^2y^3z +
        2x^2 z^4 + 2xy^4 z + 2xy^3 z^2 + 2xy^2 z^3 + y^5 z + 2y^3z^3$, subject to the following relations:
$r_{1}:=f_{1}^2 h_{2} + f_{1}f_{2}^3 + 2f_{1}f_{2}h_{1} + 2h_{1}^2,
   r_{2}:= 2f_{1}^2h_{2} + f_{1}h_{3} + 2h_{1}h_{2},
    r_{3}:=f_{1}f_{3} + 2h_{2}^2,
    r_{4}:=2f_{1}^3h_{2} + f_{1}^2f_{2}h_{2} + f_{1}^2 f_{3} + f_{2}^2h_{3} + 2f_{1}f_{2}h_{3} + f_{2}^3h_{2} + 2h_{1}h_{3},
    r_{5}:=f_{1}^2f_{3} + f_{3}h_{1} + 2h_{2}h_{3},
    r_{6}:=f_{1}^3 f_{3} + 2f_{1}f_{3}h_{1} + f_{1}f_{3}h_{2} + f_{2}^3f_{3} + 2f_{2}f_{3}h_{1} + 2h_{3}^2.$
    Thus there exists an $\F$-algebra isomorphism from $\F[x_{1},x_{2},x_{3},y_{1},y_{2},y_{3}]/(R_{1},\dots,R_{6})$ to $\F[V_{3}^{-}]^{C_{6}}$, which carries $x_{i}$ to $f_{i}$, $y_{i}$ to $h_{i}$, and $R_{j}$ to $r_{j}$, respectively.
    
We may assume that $v=(a_{1},\dots,a_{6})$ for all $a_{i}\in \F$, and suppose $a_{1}=0$. As $r_{1}(v)=r_{3}(v)=0$, thus $a_{4}=a_{5}=0$. If $a_{2}=0$, then the fact that $r_{4}(v)=0$ implies $a_{6}=0$; if $a_{2}\neq 0$, then the fact that $r_{6}(v)=0$ does also imply $a_{6}=0$. Thus $v=(0,a_{2},a_{3},0,0,0)$. Since the rank of the Jacobian matrix of the all $r_{i}$ at $v$ is not equal to
3, we see that $a_{2}=0$. Thus $\{v=(0,0,a,0,0,0)\mid a\in \F\}$ is a subset of singularities. 
$\hbo$}\end{example}

We close this section by proving the following result.

\begin{proposition}\label{prop3.5}
For $k\geq 4$, $V_k^{\pm}/C_{2p}$ are not Cohen-Macaulay.
\end{proposition}

\begin{lemma}\label{lem3.6}
If $k\geq 2$ and $f\in \F[V_k^+]^{C_{2p}}$ is any homogenous polynomial of positive degree, then
$T^-_k(f)=(-1)^{\deg(f)}f$.
\end{lemma}

\begin{proof}
 Let $\alpha=\textrm{diag}\{-1,-1,\dots,-1\}$ be the $k\times k$ diagonal  matrix. Then
 $T^-_k=\alpha\circ T^+_k$. Note that $\alpha(x_i)=-x_i$ for $1\leq i\leq k$. Thus
 $\alpha(x_1^{a_1}\cdots x_k^{a_k})=(-1)^{a_1+\dots+a_k}x_1^{a_1}\cdots x_k^{a_k}$ for any monomial $x_1^{a_1}\cdots x_k^{a_k}$. As $f$ is homogenous, $T^-_k(f)=(\alpha\circ T^+_k)\cdot f=\alpha(f)=(-1)^{\deg(f)}f$.
\end{proof}

\begin{proof}[Proof of Proposition \ref{prop3.5}]
We suppose  $\F[V_k^+]\cong\F[x_1,x_2,\dots,x_k]\cong \F[V_k^-]$. Note that $\F[V_k^+]^{C_{2p}}=\F[V_k^+]^{P}$, where $P$ is the cyclic group of order $p$, generated by $T_k^+$. Since the dimension of the image of $T_k^+-1$ is equal to $k-1$, it follows from Theorem \ref{kemper99} that $\F[V_k^+]^{P}$ is not Cohen-Macaulay for $k\geq 4$. Hence, $V_k^{+}/C_{2p}$ is not Cohen-Macaulay.

To show $V_k^{-}/C_{2p}$ is not Cohen-Macaulay, we let
$\ell_1:=x_1,\ell_2:=x_2^2-x_1(x_2+2x_3)$, $\ell_3:=x_2^3+x_1^2(3x_4-x_2)-3x_1x_2x_3$, and $N(x_2):=x_2^p-x_1^{p-1}x_2$. By Shank \cite[Theorem 4.1 and Section 2]{Sha1998}, we see that $\ell_1,\ell_2,\ell_3,N(x_{2})\in \F[V_k^+]^{P}$ and they are algebraically independent over $\F$.
Thus $\{\ell_1,\ell_2,N(x_2)\}$ is a partial homogenous system of parameters for $\F[V_k^+]^{C_{2p}}$. Moreover, we \textit{claim} that $\{\ell_1,\ell_2,N(x_2)\}$ is not a regular sequence in $\F[V_k^+]^{C_{2p}}$. In fact,  note that $p>2$ and $\ell_3\cdot N(x_2)-\ell_2^{\frac{p+3}{2}}\in (\ell_1)\F[V_k^+].$ We may assume that
$\ell_3\cdot N(x_2)-\ell_2^{\frac{p+3}{2}}=\ell_1\cdot\ell$ for some $\ell\in \F[V_k^+]$ with $\deg(\ell)>0$.
Thus $\ell\in \F(V_k^+)^{C_{2p}}\cap \F[V_k^+]=\F[V_k^+]^{C_{2p}}$, where $\F(V_k^+)^{C_{2p}}$ denotes the invariant field. This means $\ell_3\cdot N(x_2)-\ell_2^{\frac{p+3}{2}}\in (\ell_1)\F[V_k^+]^{C_{2p}}$, and
\begin{equation}\label{e3.1}
\ell_3\cdot N(x_2)\in (\ell_1,\ell_2)\F[V_k^+]^{C_{2p}}.
\end{equation}
To accomplish the proof of the claim, we need to show that $\ell_{3}$ is not in the ideal $(\ell_1,\ell_2)\F[V_k^+]^{C_{2p}}$, i.e., the image of $\ell_{3}$ in the quotient ring
$\F[V_k^+]^{C_{2p}}/(\ell_1,\ell_2)\F[V_k^+]^{C_{2p}}$ is not zero, which together with  (\ref{e3.1}) will imply that  the image of $N(x_{2})$
in $\F[V_k^+]^{C_{2p}}/(\ell_1,\ell_2)\F[V_k^+]^{C_{2p}}$ is a zerodivisor.
We use the graded lexicographic monomial order with $x_{2}>x_{1}>x_{3}>x_{4}>\cdots>x_{k}$ and use $\LM(f)$ to denote the leading monomial of a polynomial $f$. Assume by way of contradiction that $\ell_{3}=\ell_{1}f_{1}+\ell_{2}f_{2}$ for some $f_{1},f_{2}\in\F[V_k^+]^{C_{2p}}$.
Since $\LM(\ell_{3})=x_{2}^{3}>\LM(\ell_{1}f_{1})$, we have $\LM(\ell_{3})=\LM(\ell_{2}f_{2})$. Since $\LM(\ell_2)=x_2^2$, we have $\LM(f_{2})=x_{2}$.
However, an immediate fact is that  there is no element of  $\F[V_k^+]^{C_{2p}}$ with this lead monomial.
Therefore,
 $\{\ell_1,\ell_2,N(x_{2})\}$ is not a regular sequence in $\F[V_k^+]^{C_{2p}}$.

Now we are ready to show that $V_k^{-}/C_{2p}$ is not Cohen-Macaulay.
We define $ \widetilde{\ell_1}:=\ell_1^2, \widetilde{\ell_3}:=\ell_{1}\cdot\ell_3$ and $\widetilde{N}:=\ell_{1}\cdot N(x_{2})$.
By Lemma \ref{lem3.6} we see that $\widetilde{\ell_1},\ell_{2},\widetilde{\ell_3},\widetilde{N}\in \F[V_k^-]^{C_{2p}}$.
As $\widetilde{\ell_1},\ell_2,\widetilde{N}$ are algebraically independent over $\F$, they could be
a part of a homogenous system of parameters for $\F[V_k^-]^{C_{2p}}$. To show $\F[V_k^-]^{C_{2p}}$
is not Cohen-Macaulay, it suffices to show that $\{\widetilde{\ell_1},\ell_2,\widetilde{N}\}$  is not a regular sequence in
$\F[V_k^-]^{C_{2p}}$. We have seen that $\ell_3\cdot N(x_2)-\ell_2^{\frac{p+3}{2}}\in (\ell_1)\F[V_k^+]^{C_{2p}}$, so
$$\widetilde{\ell_3}\cdot \widetilde{N}-\ell_2^{\frac{p+3}{2}}\cdot\widetilde{\ell_{1}}\in (\widetilde{\ell_1})\F[V_k^-]^{C_{2p}}.$$
Namely,
$
\widetilde{\ell_3}\cdot \widetilde{N}\in (\widetilde{\ell_1},\ell_2)\F[V_k^-]^{C_{2p}}.
$
As in the previous paragraph, it is sufficient to show that  $\widetilde{\ell_{3}}$ is not in the ideal $(\widetilde{\ell_1},\ell_2)\F[V_k^-]^{C_{2p}}$.
Assume by way of contradiction that $\widetilde{\ell_{3}}=\widetilde{\ell_1}\cdot g_{1}+\ell_2\cdot g_{2}$ for $g_{1},g_{2}\in \F[V_k^-]^{C_{2p}}$. Note that $\F[V_k^-]^{C_{2p}}\subseteq \F[V_k^-]^{C_{p}}\cong\F[V_k^+]^{C_{p}}=\F[V_k^+]^{C_{2p}}$, so we could work in $\F[V_k^+]^{C_{p}}$, which is factorial; see for example Campbell-Wehlau \cite[Theorem 3.8.1]{CW2011}. Since $\ell_{1}\cdot \ell_{3}=\ell_{1}\cdot\ell_{1}\cdot g_{1}+\ell_{2}\cdot g_{2}$, it follows that $\ell_{1}$ divides $\ell_{2}\cdot g_{2}$.
As $\ell_{1}$ does not divide $\ell_{2}$, we see that $g_{2}$ must be divisible by $\ell_{1}$.
Thus $\ell_{3}=\ell_{1}\cdot g_{1}+\ell_{2}\cdot g_{2}'$ for some $g_{2}'\in \F[V_k^+]^{P}=\F[V_k^+]^{C_{2p}}$.
This contradicts with the fact that $\ell_{3}$ is not in the ideal $(\ell_1,\ell_2)\F[V_k^+]^{C_{2p}}$.
Therefore,  $\F[V_k^-]^{C_{2p}}$ is not Cohen-Macaulay, and the proof is completed.
\end{proof}

\section{\scshape Modular Quotient Varieties of $C_{2p}(p>2)$ in Characteristic 2}

In this section we suppose $p>2$, char$(\F)=2$, and $C_{2p}$ is generated by $\sigma$.
It is well-known that up to equivalence, there are only $2p$ indecomposable modular representations $\{V_{i},W_{i}\mid 0\leqslant i\leqslant p-1\}$ of $C_{2p}$, where $\dim(V_{i})=1$ and $\dim(W_{i})=2$ for all $i$. Moreover, suppose $\{1=\xi^{0},\xi,\xi^{2},\dots,\xi^{p-1}\}$ denotes the set of all roots of the polynomial $X^{p}-1=0$ in $\F$. Then  as a $C_{2p}$-representation, $V_{i}$ corresponds to the group homomorphism defined by $\sigma\mapsto \xi^{i}$. Note that $V_{0}$ is the trivial representation and these $V_{i}$ are exactly all irreducible representations of $C_{2p}$ over $\F$. For $0\leqslant i\leqslant p-1$, $W_{i}$ corresponds to the group homomorphism defined by $$\sigma\mapsto \begin{pmatrix}
   \xi^{i}   &    0\\
      1&  \xi^{i}
\end{pmatrix};$$
see Alperin \cite[pages 24--25]{Alp1993} for details.

\begin{proposition}\label{prop4.1}
Let $p>\cha(\F)=2$ be a prime and $\af_{\F}^{n}/C_{2p}$ be a reduced Cohen-Macaulay modular quotient variety.
Then there exist $a_{1},\dots,a_{p-1}\in\N$ and $b,c\in\{0,1\}$ such that
$$\af_{\F}^{n}/C_{2p}\cong (\oplus_{i=1}^{p-1} a_{i}V_{i}\oplus bW_{k}\oplus cW_{s})/C_{2p}$$
where $0\leqslant k,s\leqslant p-1$.
\end{proposition}

\begin{proof}
As a reduced modular representation of $C_{2p}$, $\af_{\F}^{n}$ is isomorphic to
$(\oplus_{i=1}^{p-1} a_{i}V_{i})\oplus (\oplus_{j=0}^{p-1} b_{j}W_{i})$ where
$a_{i},b_{j}\in\N$. Let $P$ be the Sylow 2-subgroup of $C_{2p}$, generated by $\sigma^{p}$. Since $\af_{\F}^{n}/C_{2p}$ is Cohen-Macaulay, it follows from Theorem \ref{ES} that codim$(\af_{\F}^{n})^{P}\leqslant 2$, i.e., the rank of $\sigma^{p}-1$ is less than or equal to 2. Note that the rank of $\sigma^{p}-1$ on each $V_{i}$ and $W_{j}$ are zero and one respectively.
Hence, $\af_{\F}^{n}$ is isomorphic to $\oplus_{i=1}^{p-1} a_{i}V_{i}\oplus bW_{k}\oplus cW_{s}$ for some $b,c\in\{0,1\}$ and $0\leqslant k,s\leqslant p-1$, as desired.
\end{proof}

The main purpose of this section is to study 2-dimensional reduced modular quotient varieties $\af_{\F}^{2}/C_{2p}$ and their singularities over $\F$. By Proposition \ref{prop4.1} we see that $\af_{\F}^{2}$ is isomorphic to either $V_{i}\oplus V_{j}$ or $W_{k}$ for some $1\leqslant i,j\leqslant p-1$ and $0\leqslant k\leqslant p-1$.
This leads us to calculate two families of invariant rings $\F[V_{i}\oplus V_{j}]^{C_{2p}}$ and $\F[W_k]^{C_{2p}}$.

The following result due to Harris-Wehlau \cite[Section 3]{HW2013} has already exhibited a minimal generating set for the invariant ring $\F[V_{i}\oplus V_{j}]^{C_{2p}}$ and all relations among this generating set.

\begin{proposition}\label{HWp}
For any $1\leqslant i,j\leqslant p-1$, there exists $m\in\N^+$ such that the invariant ring $\F[V_{i}\oplus V_{j}]^{C_{2p}}=\F[x,y]^{C_{p}}$ is minimally generated by
$\{u_{k}:=x^{a_{k}}y^{b_{k}}\mid k=1,2,\dots,m\}$
where $p=a_{1}>a_{2}>\cdots>a_{m}=0$ and $0=b_{1}<b_{2}<\cdots <b_{m}=p$. Moreover, let
$$\pi:\F[U_{1},U_{2},\dots,U_{m}]\ra\F[x,y]^{C_{p}},\quad\textrm{ by~~ } U_{k}\mapsto u_{k}$$
be the standard, surjective $\F$-algebra homomorphism, where $\F[U_{1},U_{2},\dots,U_{m}]$ denotes a polynomial algebra. Then there exist $m-1\choose 2$ elements
$$\{R_{kt}\mid 1\leqslant k,t\leqslant m\textrm{ and }t-k\geqslant 2\}$$
minimally generate $\ker(\pi)$, where
$R_{kt}:=U_{k}U_{t}-U_{k+1}\prod_{s=k+1}^{t-1} U_{s}^{d_{kts}}$
for some $d_{kts}\in\N^{+}$.
\end{proposition}

\begin{remark}\label{rem4.3}
{\rm
The invariant ring $\F[V_{i}\oplus V_{j}]^{C_{2p}}$ is not necessarily a complete intersection. For example,
if $p=5$, Magma calculation shows that $\F[V_{1}\oplus V_{2}]^{C_{10}}$ is minimally generated by four elements $\{xy^{2},x^{3}y,x^{5},y^{5}\}$, subject to three relations $\{R_{13},R_{14},R_{24}\}$.
$\hbo$}\end{remark}

\begin{remark}\label{rem4.4}
{\rm
The invariant ring $\F[V_{i}\oplus V_{j}]^{C_{2p}}$ is not necessarily Gorenstein in general. For example,
if $p=5$, then the image of $C_{10}$ on $V_{1}\oplus V_{2}$ consists of $$\left\{\begin{pmatrix}
     1 & 0   \\
      0&1
\end{pmatrix},\begin{pmatrix}
    \xi  & 0   \\
    0  &  \xi^{2}
\end{pmatrix},\begin{pmatrix}
    \xi^{2}  & 0   \\
    0  &  \xi^{4}
\end{pmatrix}, \begin{pmatrix}
    \xi^{3}  & 0   \\
    0  &  \xi
\end{pmatrix},\begin{pmatrix}
    \xi^{4}  & 0   \\
    0  &  \xi^{3}
\end{pmatrix}\right\}$$
in which there are no reflections. Assume that $\F[V_{1}\oplus V_{2}]^{C_{10}}$ is Gorenstein.  By Watanabe \cite[Theorem 1]{Wat1974b} we see that the image of $C_{10}$ on $V_{1}\oplus V_{2}$ is contained in $\SL(2,\F)$. This is a contradiction.
$\hbo$}\end{remark}

Recall that $\F[V_{i}\oplus V_{j}]^{C_{2p}}$ is always Cohen-Macaulay; see Smith \cite[Theorem 1.2]{Smi1996}. The example in Remark \ref{rem4.4} leads us to consider when $\F[V_{i}\oplus V_{j}]^{C_{2p}}$ is Gorenstein. Here we prove Theorem \ref{thm1.4}.

\begin{proof}[Proof of Theorem \ref{thm1.4}]
$(\RA)$ Assume by way of contradiction that $i+j\neq p$. The matrix corresponding to the generator $\sigma$ of $C_{2p}$ on
$V_{i}\oplus V_{j}$ is $\begin{pmatrix}
    \xi^{i}  &    0\\
     0 &  \xi^{j}
\end{pmatrix}$, which has determinate $\xi^{i+j}\neq 1$. This indicates that the image of $C_{2p}$ on $V_{i}\oplus V_{j}$ is not contained in $\SL(2,\F)$. By Watanabe \cite[Theorem 1]{Wat1974b}, the fact that the image of $C_{2p}$ contains no reflections, together with the assumption that $\F[V_{i}\oplus V_{j}]^{C_{2p}}$ is Gorenstein, implies that it must be contained in $\SL(2,\F)$. This contradiction shows that $i+j=p$.
$(\LA)$ Assume that $i+j=p$. Then $\det(\sigma)=\xi^{i+j}=\xi^{p}=1$. Thus the image of $C_{2p}$ is contained in $\SL(2,\F)$.
By Watanabe \cite[Theorem 1]{Wat1974a} (or Braun \cite[Theorem]{Bra2011}) we see that $\F[V_{i}\oplus V_{j}]^{C_{2p}}$ is Gorenstein.
\end{proof}

Now we consider the invariant ring $\F[W_{k}]^{C_{2p}}$ for $k=0,1,\dots,p-1$. Clearly, the action of $C_{2p}$ on $W_{0}$ is not faithful, and actually, in this case, $\F[W_{0}]^{C_{2p}}=\F[x,y]^{C_{2}}=\F[x,xy+y^{2}]$ is a polynomial algebra. Note that the actions of $C_{2p}$ on $W_{k}$ are faithful for $k=1,2,\dots,p-1$.

Fix $k\in \{1,2,\dots,p-1\}$.  Suppose $\sigma=\begin{pmatrix}
    \xi^{k}  &  0  \\
    1  &  \xi^{k}
\end{pmatrix}$ denotes a generator of $C_{2p}$ on $W_{k}$. Then $\sigma^{p}$ generates a subgroup of $C_{2p}$ of order 2, say $H$. Further, we have seen that $\F[W_{k}]^{H}=\F[x,y]^{C_{2}}=\F[h_1,h_2]$ is a polynomial algebra, where $h_1:=x$ and $h_2:=xy+y^{2}$. Thus,
$$\F[W_{k}]^{C_{2p}}=(\F[W_{k}]^{H})^{C_p}\cong \F[h_1,h_2]^{C_p}$$
where $C_p$ can be generated by $\tau:=\begin{pmatrix}
    \xi  &  0  \\
    0 &  \xi
\end{pmatrix}$, i.e., $\tau(h_1)=\xi^{-1}\cdot h_1$ and $\tau(h_2)=\xi^{-2}\cdot h_2$.
Therefore, $\F[W_{k}]^{C_{2p}}\cong\F[h_1,h_2]^{C_p}\cong \F[V_1\oplus V_2]^{C_p}\cong \F[V_1\oplus V_2]^{C_{2p}}.$

Applying Proposition \ref{HWp}, we derive

\begin{proposition}\label{prop4.6}
For $1\leqslant k\leqslant p-1$, there exists $m\in \N^+$ such that $\F[W_k]^{C_{2p}}=\F[x,y]^{C_{2p}}$ is minimally generated by
$\{f_{i}:=h_1^{a_{i}}h_2^{b_{i}}\mid i=1,2,\dots,m\}$
where $p=a_{1}>a_{2}>\cdots>a_{m}=0$ and $0=b_{1}<b_{2}<\cdots <b_{m}=p$. Moreover, let
$$\pi:\F[U_{1},U_{2},\dots,U_{m}]\ra\F[x,y]^{C_{p}},\quad \textrm{ by }U_{i}\mapsto f_{i}$$
be the standard, surjective $\F$-algebra homomorphism, where $\F[U_{1},U_{2},\dots,U_{m}]$ denotes a polynomial algebra. Then there exist $m-1\choose 2$ elements
$$\{R_{ij}\mid 1\leqslant i,j\leqslant m\textrm{ and }j-i\geqslant 2\}$$
minimally generate $\ker(\pi)$, where
$R_{ij}:=U_{i}U_{j}-U_{i+1}\prod_{s=j+1}^{j-1} U_{s}^{d_{ijs}}$
for some $d_{ijs}\in\N^{+}$.
\end{proposition}

\begin{example}\label{counter}
{\rm For $p=3$, the invariant ring $\F[W_k]^{C_{6}}$ is  minimally generated by
$$\{f_1=x^3,f_2=x^2y+xy^2,f_3=(xy+y^2)^3\}$$
subject to one relation: $f_2^3+f_1f_3=0.$ Thus $\F[W_k]^{C_{6}}$ is a hypersurface.
$\hbo$}\end{example}

\begin{remark}\label{rem4.8}
{\rm
By Remarks \ref{rem4.3} and \ref{rem4.4} we see that $\F[W_k]^{C_{2p}}$ is not necessarily a complete intersection or Gorenstein. Further, it follows from Theorem \ref{thm1.4} that
$\F[W_k]^{C_{2p}}$ is never Gorenstein unless $p=3$.
$\hbo$}\end{remark}

\section{\scshape Wild McKay Correspondence}

In this last section we will show that Problem \ref{prob1.1} have a negative answer for the modular cases of $C_{2p}$ if the condition that ``$G$ is a small subgroup'' was dropped.

To begin with, we first recall the definition of the Grothendieck ring of varieties. 

\begin{definition}{\rm
Let $\var_{\F}$ be the set of isomorphism classes of $\F$-varieties. The \textbf{Grothendieck ring of varieties} over $\F$, written $K_0(\var_{\F})$, is the abelian group generated by $[Y]\in\var_{\F}$ subject to the following relation: If $Z$ is a closed subvariety of $Y$, then $[Y]=[Y\backslash Z]+[Z]$. It has a ring structure where the product is defined by $[Y]\cdot[Z]:=[Y\times Z]$. We denote by $\LL$ the class [$\af_{\F}^1$] of the affine line.
}\end{definition}

The McKay correspondence has many versions for different subjects of study  in the literature.
Let's recall the version of Batyrev \cite{Bat1999}  in terms of stringy invariants and over $\C$. Suppose $G$ is a finite subgroup of $\SL_n(\C)$ and $X$ denotes the associated quotient variety $\af_{\C}^n/G$. 
The McKay correspondence reveals the following equality in a certain modification of  $K_0(\var_{\C})$:
\begin{equation}\label{McKayC}
M_{\st}(X)=\sum_{g\in \conj(G)}\LL^{\age(g)},
\end{equation}
where $M_{\st}(X)$ denotes the stringy motivic invariant of $X$ defined by a resolution, $\conj(G)$ denotes the set of conjugacy classes of $G$ and $\age(g)$ denotes the \textbf{age} of $g$. Here we note that $\age(g)$ is determined by eigenvalues of elements in $g$ and thus it is independent of the choice of the representative in $g$. Moreover, if $Y\ra X$ is a crepant resolution, then $M_{\st}(X)=[Y]$; see Reid \cite{Rei2002} for the details. 
Without using resolutions of $X$, Denef and Loeser \cite{DL2002} also redefined $M_{\st}(X)$ by the theory of motivic integration over the arc space of $X$, which is workable not only over $\C$ but also  over arbitrary algebraic closed fields. 
Yasuda \cite{Yas2014} conjectured and proved for the case of $C_{p}$ that  the right hand side of (\ref{McKayC}) could be changed to some integration over the moduli space of $G$-covers. As a consequence, Problem \ref{prob1.1} has positive answer for the case of $C_{p}$.

To calculate the element $[Y]$ in the Grothendieck ring, we need the following result whose proof can be found in Hartmann
 \cite[Proposition 6.2]{Har2016}.


\begin{proposition}\label{WQ}
Let $G$ be a finite abelian group with quotient $G\ra \Gamma$. Let $k$ be
a field of characteristic $p$, $q$ be the greatest divisor of $|G|$ prime to $p$, and let
$K/k$ be a Galois extension with Galois group $\Gamma$. Assume that the Galois action on
$K$ lifts to a $k$-linear action of $G$ on a finite dimensional $K$-vector space $V$. If $k$
contains all $q$-th roots of unity, then
$$[V/G]=\LL_k^{\dim_K V}\in K_0(\var_k).$$
\end{proposition}




We conclude this paper with the following answer to Problem \ref{prob1.1} provided that the condition that ``$G$ is a small subgroup'' was dropped.

\begin{proof}[An Example Attached to Problem \ref{prob1.1}]
Let's go back to Example \ref{counter} where we take the ground field $\F$ to be 
an algebraic closed field of characteristic $2$, which contains all third roots of unity. 

We observe that the origin is the unique
 singular locus in the quotient variety. Blowing up of the origin, we derive a crepant resolution $Y\ra X$. The exceptional set is a simple normal crossing divisor with two irreducible components, say $E_1\cup E_2$, where both $E_1$ and $E_2$ are $\mathbb{P}^1_{\F}$. Hence, 
\begin{eqnarray*}
[Y]&= &[Y\backslash \{E_1\cup E_2\}]+[E_1\cup E_2] \\
&= & [X\backslash\{0\}]+2(\LL+1)-1\\
&= & [X]+2\LL\\
&= & \LL^2+2\LL,
\end{eqnarray*}
where the last equality follows from Proposition \ref{WQ}. Therefore,  the Euler characteristic of $Y$ is equal to 3 while the the cardinality of the conjugacy classes of $C_{6}$ is 6. This shows that Problem \ref{prob1.1} fails to be valid for the case of $C_{6}$.
\end{proof}

\section*{Acknowledgments}

This research is supported by the National Natural Science Foundation of China (Grant No. 11471116, 11531007), Science and Technology Commission of Shanghai Municipality (Grant No. 18dz2271000) and SAFEA (P182009020).
The authors would like to thank Liang Chen and Yongjiang Duan for their support. 



\end{document}